\documentclass{amsart}

\newtheorem{theorem}{Theorem}[section]
\newtheorem{lemma}[theorem]{Lemma}

\theoremstyle{definition}
\newtheorem{definition}[theorem]{Definition}
\newtheorem{example}[theorem]{Example}

\theoremstyle{remark}
\newtheorem{remark}[theorem]{Remark}

\numberwithin{equation}{section}

\begin{document}

\title[Warped product skew semi-invariant submanifolds]{Warped product skew semi-invariant\\
submanifolds of order $1$ of a locally\\
product Riemannnian manifold}

\author{Hakan Mete Ta\c stan}

\address{\.Istanbul University\\
Department of Mathematics\\
Vezneciler, \.Istanbul, Turkey}

\email{hakmete@istanbul.edu.tr}


\begin{abstract}
We introduce warped product skew semi-invariant submanifolds of
order $1$ of a locally product Riemannian manifold. We give a
necessary and sufficient condition for skew semi-invariant
submanifold of order 1 to be a locally warped product. We also prove
that the invariant distribution which is involved in the definition
of the submanifold is integrable under some restrictions. Moreover,
we find an inequality between the warping function and the squared
norm of the second fundamental form for such submanifolds. Equality
case is also discussed.
\end{abstract}
\maketitle
\section{introduction}
The theory of submanifolds is one of the most popular research
area in differential geometry. In an almost Hermitian manifold, its almost
complex structure determines several types  of submanifolds.
For example, holomorphic (invariant) submanifolds and totally real (anti-invariant)
submanifolds are determined by the behavior of the almost complex structure.
In the first case the tangent space of the submanifolds is invariant under
the action of the almost complex structure. In the second case the tangent space of
the submanifolds is anti-invariant, that is , it is mapped into the normal space.
A. Bejancu \cite{Be} introduced the notion of CR-submanifolds of a K\"{a}hlerian
manifold as a natural generalization of invariant and anti-invariant submanifolds.
A CR-submanifold is said to be proper if it is neither invariant nor  anti-invariant. The theory of CR-submanifolds has
been a most interesting topics since then. Slant submanifolds are another generalization
of invariant and anti-invariant submanifolds. This type submanifolds is defined by
B.Y. Chen \cite{Che}. Since then such submanifolds have been studied by many geometers
(see \cite{Ar,Ca,Lo} and references therein). If a slant submanifold is neither invariant
nor anti-invariant then it is said to be proper. We observe that a proper CR-submanifold
is never a slant submanifold. In \cite{Papa}, N. Papaghiuc introduced the notion
of semi-slant submanifolds obtaining CR-submanifolds and  slant submanifolds as special cases.
A. Carriazo \cite{Ca}, introduced bi-slant submanifolds which is a generalization of
semi-slant submanifolds. One of the classes of such submanifolds is that of anti-slant
submanifolds. This type submanifolds are also generalization of slant and CR-submanifolds.
However, B. \c Sahin  \cite{Sa} called these submanifolds as hemi-slant submanifolds
because of that the name anti-slant seems to refer that it has no slant factor.
He also observed that there is no inclusion between proper hemi-slant submanifolds
and proper semi-slant submanifolds. We note that hemi-slant submanifolds
are also studied under the name of pseudo-slant submanifolds (see \cite{Kh,Ud}).\\

Skew CR-submanifolds of a K\"{a}hlerian manifold are first defined by G.S. Ronsse in \cite{Ro}.
Such submanifolds are a generalizations of bi-slant submanifolds. Consequently, invariant, anti-invariant,
CR, slant, semi-slant and hemi-slant submanifolds are particular cases of skew CR-submanifolds. We notice that
CR-submanifolds in K\"{a}hlerian manifolds correspond to semi-invariant submanifolds \cite{Bej}
in locally product Riemannian manifolds. Therefore, skew CR-submanifolds in K\"{a}hlerian manifolds
correspond to skew semi-invariant submanifolds in locally product Riemannian manifolds.
For the fundamental properties and further studies of skew CR-submanifolds; see \cite{Ro} and \cite{Trip}.
Skew semi-invariant submanifolds of a locally product Riemannian manifold were studied first
by X. Liu and F.-M. Shao in \cite{Li}.\\

The notion of warped product was initiated by R.L. Bishop and B.
O'Neill \cite{Bi}. Let $M_{1}$ and $M_{2}$ be two Riemanian
manifolds with Riemannian metrics $g_{1}$ and $g_{2}$ respectively.
Let $f$ be positive differentiable function on $M_{1}$. The warped
product $M=M_{1}\times _{f}M_{2}$ of $M_{1}$ and $M_{2}$ is the
Riemannian manifold $(M_{1}\times M_{2},g)$, where
\begin{equation}\label{}\nonumber
\begin{array}{c}
g=g_{1}+f^{2}g_{2}
\end{array}.
\end{equation}
More explicitly, if $U\in T_{p}M$, then
\begin{equation}\label{}\nonumber
\begin{array}{c}
\|U\|^{2}=\|d\pi_{1}(U)\|^{2}+(f^{2}\circ\pi_{1})\|d\pi_{2}(U)\|^{2}
\end{array},
\end{equation}
where $\pi_{i}, i=1,2,$ are the canonical projections $M_{1}\times
M_{2}$ onto $M_{1}$ and $M_{2}$ respectively. The function $f$ is
called the\emph{ warping function} of the warped product. If the
warping function is constant, then the manifold $M$ is said to be
\emph{trivial}. It is also known that $M_{1}$ is totally geodesic
and $M_{2}$ is totally umbilical from \cite{Bi}. For a warped
product $M_{1}\times _{f}M_{2}$, we denote by $\mathcal{D}_{1}$ and
$\mathcal{D}_{2}$ the distributions given by the vectors tangent to
leaves and fibers respectively. Thus, $\mathcal{D}_{1}$ is obtained
from tangent vectors to $M_{1}$ via horizontal lift and
$\mathcal{D}_{2}$  is obtained by tangent vectors of $M_{2}$ via
vertical lift. Let $U$ be a vector field on $M_{1}$ and $V$ be vector
field on $M_{2}$, then from Lemma 7.3 of \cite{Bi}, we have
\begin{equation}\label{e3}
\begin{array}{c}
\nabla_{U}V=\nabla_{V}U=U(\ln f)V
\end{array},
\end{equation}
where $\nabla$ is the Levi-Civita connection on $M_{1}\times
_{f}M_{2}$.\\

Warped product submanifolds have been studying very actively since
B.Y. Chen \cite{Chenn} introduced the notion of CR-warped product in
K\"{a}hlerian manifolds. In fact, different type warped product
submanifolds of different kinds structures are studied last thirteen
years. For example; see \cite{Al,Ha,S2,S3,Sa,Sah,Ud}. Most of the
studies related to this topic can be found in the survey book
\cite{Chen2}. Recently, B. \c Sahin \cite{Sah} introduced the notion
of skew CR-warped product submanifolds of K\"{a}hlerian manifolds
which is a generalization of different kind warped product
submanifolds studied by many authors. We note that warped product
skew CR-submanifolds of a cosymplectic manifold were
studied in \cite{Ha}.\\

In this paper, we define and study warped product skew
semi-invariant submanifolds of order $1$ of a locally product
Riemannian manifold. We give an illustrate example and prove a
characterization theorem for the mixed totally geodesic proper skew
semi-invariant submanifold using some lemmas. In general, the
invariant distribution of a submanifold is not integrable in a
locally product Riemannian manifold. However, we prove that the
invariant distribution of a warped product skew semi-invariant
submanifold of order $1$ is integrable in a locally product
Riemannian manifold under some restrictions. Finally, we obtain an
inequality between the warping function and the squared norm of the
second fundamental form for such submanifolds. Equality case is also
considered.

\section{preliminaries}
Let $(\bar{M},g,F)$ be a locally product Riemannian manifold or, (briefly, l.p.R. manifold).
It means that \cite{Yan} $\bar{M}$ has a tensor field $F$ of type $(1,1)$ on $\bar{M}$
such that, $\forall  \bar{U}, \bar{V}\in T\bar{M}$, we have
\begin{equation}
\label{e4}
\begin{array}{c}
F^{2}=I, (F\neq\pm I), \quad g(F\bar{U},F\bar{V})=g(\bar{U},\bar{V})\quad and  \quad(\overline{\nabla}_{\bar{U}}\,\,F)\bar{V}=0
\end{array},
\end{equation}
where $g$ is the Riemannian metric, $\overline{\nabla}$ is the Levi-Civita connection on $\bar{M}$
and $I$ is the identifying operator on the tangent bundle $T\bar{M}$ of $\bar{M}$.\\

Let $M$ be a submanifold of a l.p.R. manifold $(\bar{M},g,F)$ as an isometrically immersed.
Let  ${\nabla}$ and $\nabla^{\bot}$ be the induced, and induced normal connection in $M$ and the
normal bundle $T^{\bot}M$ of $M$, respectively. Then for all  $U,V\in TM$ and $\xi\in T^{\bot}M$
the Gauss and Weingarten formulas are given by
\begin{equation}
\label{e5}
\begin{array}{c}
\overline{\nabla}_{U}V={\nabla}_{U}V+h(U,V)
\end{array}
\end{equation}
and
\begin{equation}
\label{e6}
\begin{array}{c}
\overline\nabla_{U}\xi=-A_{\xi}U+\nabla_{U}^{\bot}\xi
\end{array}
\end{equation}
where $h$ is the \emph{second fundamental form} of $M$ and $A_{\xi}$ is the Weingarten endomorphism
associated with $\xi.$ The second fundamental form $h$ and the \emph{shape operator} $A$ related by

\begin{equation}
\label{e7}
\begin{array}{c}
g(h(U,V),\xi)=g(A_{\xi}U,V)
\end{array}.
\end{equation}
The \emph{mean curvature vector field} $H$ is given by
$H=\frac{1}{m}(trace\,h),$ where $dim(M)=m.$ The submanifold $M$ is
called \emph{totally geodesic} in $\bar{M}$ if $h=0$ and
\emph{minimal} if $H=0.$ If $h(U,V)=g(U,V)H$ for all $U, V\in TM$,
then $M$ is \emph{totally umbilical}.

\section{skew semi-invariant submanifolds of order $1$\\
of a locally product Riemannian manifold}
Let $\bar{M}$ be a l.p.R. manifold with Riemannian metric $g$ and almost product structure $F.$
Let $M$ be Riemannian submanifold isometrically immersed in  $\bar{M}$. For any $U\in TM$, we write
\begin{equation} \label{e9}
\begin{array}{c}
FU=TU+NU
\end{array}.
\end{equation}
Here $TU$ is the tangential part of $FU,$ and
$NU$ is the normal part of $FU.$ Similarly, for any $\xi\in
T^{\bot}M$, we put
\begin{equation}
\label{e10}
\begin{array}{c}
F\xi=t\xi+\omega\xi
\end{array},
\end{equation}
where $t\xi$ is the tangential part of $F\xi,$ and $\omega\xi$ is
the normal part of $F\xi.$\\

Using (\ref{e4}) and (\ref{e9}), we have $g(T^{2}U,V)=g(T^{2}V,U)$
for all $U, V\in TM$. It says that $T^{2}$ is a symmetric operator
on the tangent space $T_{p}M, p\in M$. Therefore its eigenvalues are
real and diagonalizable. Moreover, its eigenvalues are bounded by
$0$ and $1.$ For each $p\in M$, we set
$$\mathcal{D}_{p}^{\lambda}=Ker\{T^{2}-\lambda^{2}(p)I\}_{p}\,,$$
where $I$ is the identity endomorphism and $\lambda(p)$ belongs to
closed interval $[0,1]$ such that $\lambda^{2}(p)$ is an eigenvalue
of $T_{p}^{2}$. Since $T_{p}^{2}$ is symmetric and diagonalizable,
there is some integer $k$ such that
$\lambda_{1}^{2}(p),...,\lambda_{k}^{2}(p)$ are distinct eigenvalues
of $T_{p}^{2}$ and $T_{p}M$ can be decomposed as a direct sum of
mutually orthogonal eigenspaces, i.e.
$$T_{p}M=\mathcal{D}_{p}^{\lambda_{1}}\oplus...\oplus\mathcal{D}_{p}^{\lambda_{k}}.$$
For $i\in\{1,...,k\}$, $\mathcal{D}_{p}^{\lambda_{i}}$ is a
$T$-invariant subspace of $T_{p}M$. We note that
$\mathcal{D}_{p}^{0}=Ker T_{p}$ and $\mathcal{D}_{p}^{1}=Ker N_{p}.$
$\mathcal{D}_{p}^{0}$ is the maximal anti $F$-invariant subspace of
$T_{p}M$ where as $\mathcal{D}_{p}^{1}$ is the maximal $F$-invariant
subspace of $T_{p}M$. We denote the distributions $\mathcal{D}^{0}$
and $\mathcal{D}^{1}$ by $\mathcal{D}^\bot$ and  $\mathcal{D}^T$,
respectively from now on.

\begin{definition}  (\cite{Li})
Let $M$ be a submanifold of a l.p.R. manifold $\bar{M}$. Then $M$ is
said to be a \emph{generic submanifold} if there exists an integer
$k$ and functions $\lambda_{i}, i\in\{1,...,k\}$ defined on $M$ with
values in $(0,1)$ such that\\

\textbf{(i)} Each $\lambda_{i}^{2}(p), i\in\{1,...,k\}$ is a
distinct eigenvalue of $T_{p}^{2}$ with
$$T_{p}M=\mathcal{D}_{p}^\bot\oplus\mathcal{D}_{p}^T\oplus\mathcal{D}_{p}^{\lambda_{1}}\oplus...\oplus\mathcal{D}_{p}^{\lambda_{k}}$$
for $p\!\in\!M$.\\

\textbf{(ii)} The dimension of $\mathcal{D}^\bot$, $\mathcal{D}^T$
and $\mathcal{D}^{\lambda_{i}}, 1\leq i\leq k$ are independent of
$p\!\in\!\!M$. Moreover, if each $\lambda_{i}$ is constant on $M$,
then we say that $M$ is a \emph{skew semi-invariant submanifold} of
$\bar{M}$.
\end{definition}

In view of Definition 3.1, we observe that the following special
cases.\\

Let $M$ be a skew semi-invariant submanifold of a l.p.R. manifold
$\bar{M}$ as in Definition 3.1. Then\\

\textbf{(a)} If $k\!=\!0$ and $\mathcal{D}^\bot\!=\!\{0\}$, then $M$
is an invariant submanifold \cite{A}.\\

\textbf{(b)} If $\!k=0\!$ and $\mathcal{D}^T\!=\!\{0\}$, then $M$ is
an
anti-invariant submanifold \cite{A}.\\

\textbf{(c)} If $k\!=\!0$, then $M$ is a semi-invariant submanifold \cite{Bej}.\\

\textbf{(d)} If $\mathcal{D}^\bot\!=\!\{0\}\!=\!\mathcal{D}^T$ and
$k\!=\!1$,
then $M$ is a slant submanifold \cite{S}.\\

\textbf{(e)} If $\mathcal{D}^\bot\!=\!\{0\},
\mathcal{D}^T\!\neq\!\{0\}$ and
$k\!=\!1$, then $M$ is a semi-slant submanifold \cite{S}.\\

\textbf{(f)} If $\mathcal{D}^T\!=\!\{0\},
\mathcal{D}^\bot\!\neq\!\{0\}$ and
$k\!=\!1$, then $M$ is a hemi-slant submanifold \cite{Ta}.\\

\textbf{(g)} If $\mathcal{D}^\bot\!\!=\!\{0\}\!=\!\mathcal{D}^T$
and $k\!=\!2$, then $M$ is a bi-slant submanifold \cite{Ca}.
\begin{definition}
A submanifold $M$ of a l.p.R. manifold $\bar{M}$ is called a
\emph{skew semi-invariant submanifold of order $1$}, if $M$ is a
skew semi-invariant submanifold with $k\!=\!1$.
\end{definition}
In this case, we have
\begin{equation}\label{e11}
\begin{array}{c}
TM=\mathcal{D}^\bot\oplus\mathcal{D}^T\oplus\mathcal{D}^\theta
\end{array},
\end{equation}
where $\mathcal{D}^\theta=\mathcal{D}^{\lambda_{1}}$ and $\lambda_{1}$ is constant.
We say that a skew semi-invariant submanifold of order $1$ is\emph{ proper},
if $\mathcal{D}^\bot\!\neq\!\{0\}$ and $\mathcal{D}^T\!\neq\!\{0\}$.\\

A slant submanifold $M$ of a l.p.R. manifold $\bar{M}$ is
characterized by
\begin{equation}\label{e12}
\begin{array}{c}
T^{2}U=\lambda U
\end{array}
\end{equation}
such that $\lambda\in[0,1]$, where $U\in TM$, for details;
see\cite{S}. Moreover, if $\theta$ is the slant angle of $M$, then
we have $\lambda=\cos^{2}\!\theta.$ On the other hand, for any slant
submanifold $M$ of a l.p.R. manifold $\bar{M}$, we have
\begin{eqnarray}\label{e13}
&&(a)\quad T^{2}+tN=I, \quad\quad (b)\quad \omega^{2}+Nt=I,\nonumber\\
&&(c)\quad NT+\omega N=0,~~~~\quad(d)\quad Tt+t\omega=0.
\end{eqnarray}
For the proof of (\ref{e13}); see \cite{Ta}.\\

Throughout this paper, the letters $V,W$ will denote the vector fields of
the anti-invariant distribution $\mathcal{D}^\bot$, $U,Z$ will denote the vector fields of
the slant distribution $\mathcal{D}^\theta$ and $X,Y$ will denote the vector fields of
the invariant distribution $\mathcal{D}^T$.\\

For the further study of skew semi-invariant submanifold of order $1$ of a l.p.R. manifold,
we need to following lemmas.

\begin{lemma} Let $M$ be a proper skew semi-invariant submanifold of order $1$ of a l.p.R. manifold $\bar{M}.$
Then we have,
\begin{equation}\label{e14}
\begin{array}{c}
g(\nabla_{V}W,X)=-g(A_{FW}V,FX)
\end{array},
\end{equation}
\begin{equation}\label{e15}
\begin{array}{c}
g(\nabla_{V}Z,X)=-\csc^{2}\!\theta \{g(A_{NTZ}V,X)+g(A_{NZ}V,FX)\}
\end{array},
\end{equation}
\begin{equation}\label{e16}
\begin{array}{c}
g(\nabla_{Z}V,X)=-g(A_{FV}Z,FX)
\end{array},
\end{equation}
for $V,W\in\mathcal{D}^\bot, Z\in\mathcal{D}^\theta$ and $X\in\mathcal{D}^T$.
\end{lemma}
\begin{proof} Using (\ref{e5}) and (\ref{e4}), we have $g(\nabla_{V}W,X)=g(\overline{\nabla}_{V}FW,FX)$
for $V,W\in\mathcal{D}^\bot$ and $X\in\mathcal{D}^T$. Hence, using
(\ref{e6}), we get (\ref{e14}). In a similar way, we have
$g(\nabla_{V}Z,X)=g(\overline{\nabla}_{V}FZ,FX)$, where
$Z\in\mathcal{D}^\theta$. Then using (\ref{e9}) and (\ref{e4}), we
obtain
$g(\nabla_{V}Z,X)=g(\overline{\nabla}_{V}FTZ,X)+g(\overline{\nabla}_{V}NZ,FX)$.
Hence, using (\ref{e9}) and (\ref{e6}), we arrive at
$g(\nabla_{V}Z,X)=g(\overline{\nabla}_{V}T^{2}Z,X)+g(\overline{\nabla}_{V}N(TZ),FX)-g(A_{NZ}V,FX)$.
With the help of (\ref{e12}), (\ref{e5}) and (\ref{e6}), we get
(\ref{e15}). Similarly, one can obtain (\ref{e16}).
\end{proof}

\begin{lemma} Let $M$ be a proper skew semi-invariant submanifold of order $1$ of a l.p.R. manifold $\bar{M}.$
Then we have,
\begin{equation}\label{e17}
\begin{array}{c}
g(\nabla_{U}Z,X)=-\csc^{2}\!\theta \{g(A_{NTZ}U,X)+g(A_{NZ}U,FX)\}
\end{array},
\end{equation}
\begin{equation}\label{e18}
\begin{array}{c}
g(\nabla_{X}Y,Z)=\csc^{2}\!\theta \{g(A_{NTZ}X,Y)+g(A_{NZ}X,FY)\}
\end{array},
\end{equation}
\begin{equation}\label{e19}
\begin{array}{c}
g(\nabla_{X}Y,V)=g(A_{FV}X,FY)
\end{array},
\end{equation}
for $X,Y\in\mathcal{D}^T$, $U,Z\in\mathcal{D}^\theta$  and $V\in\mathcal{D}^\bot$.
\end{lemma}
\begin{proof} Let $U,Z\in\mathcal{D}^\theta$ and  $X\in\mathcal{D}^T$. Then using (\ref{e5}), (\ref{e4}) and (\ref{e9}),
we have
$g(\nabla_{U}Z,X)=g(\overline{\nabla}_{U}FZ,FX)=g(\overline{\nabla}_{U}TZ,FX)+g(\overline{\nabla}_{U}NZ,FX)$
Again, using (\ref{e4}) and (\ref{e6}), we obtain
$g(\nabla_{U}Z,X)=g(\overline{\nabla}_{U}FTZ,X)-g(A_{NZ}U,FX)$.
Here, if we use (\ref{e13})-(a) and (\ref{e12}), then we get
$g(\nabla_{U}Z,X)=\cos^{2}\!\theta
g(\nabla_{U}Z,X)+g(\overline{\nabla}_{U}NTZ,X)-g(A_{NZ}U,FX)$. After
some calculation, we find (\ref{e17}). For the proof of (\ref{e18}),
using (\ref{e5}), (\ref{e4}) and (\ref{e9}), we have
$g(\nabla_{X}Y,Z)=g(\overline{\nabla}_{X}FY,FZ)=g(\overline{\nabla}_{X}FY,TZ)+g(\overline{\nabla}_{X}FY,NZ)$
for $X,Y\in\mathcal{D}^T$ and  $Z\in\mathcal{D}^\theta$. Again,
using (\ref{e4}) and (\ref{e6}), we obtain
$g(\nabla_{X}Y,Z)=g(\overline{\nabla}_{X}Y,FTZ)+g(h(X,FY),NZ)$. With
the help of (\ref{e13})-(a) and (\ref{e12}), we get
$g(\nabla_{X}Y,Z)=\cos^{2}\!\theta g(\nabla_{X}Y,Z)+
g(\overline{\nabla}_{X}Y,NTZ)+g(h(X,FY),NZ)$. Upon direct
calculation, we find (\ref{e18}). In a similar way, we can obtain
(\ref{e19}).
\end{proof}

\begin{lemma} Let $M$ be a proper skew semi-invariant submanifold of order $1$ of a l.p.R. manifold $\bar{M}.$
Then we have,
\begin{equation}\label{e20}
\begin{array}{c}
g(\nabla_{V}X,Z)=\csc^{2}\!\theta \{g(A_{NTZ}V,X)+g(A_{NZ}V,FX)\}
\end{array},
\end{equation}
\begin{equation}\label{e21}
\begin{array}{c}
g(\nabla_{U}Z,V)=\sec^{2}\!\theta \{g(A_{FV}U,TZ)+g(A_{NTZ}U,V)\}
\end{array},
\end{equation}
\begin{equation}\label{e22}
\begin{array}{c}
g(\nabla_{X}V,Z)=\sec^{2}\!\theta \{g(A_{FV}X,TZ)+g(A_{NTZ}X,V)\}
\end{array},
\end{equation}
for $X\in\mathcal{D}^T$, $U,Z\in\mathcal{D}^\theta$  and
$V\in\mathcal{D}^\bot$.
\end{lemma}
\begin{proof} Using (\ref{e5}), (\ref{e4}) and (\ref{e9}),
we have\\
$g(\nabla_{V}X,Z)\!=\!g(\overline{\nabla}_{V}FX,FZ)\!=\!-g(\overline{\nabla}_{V}FZ,FX)\!=\!-g(\overline{\nabla}_{V}TZ,FX)-g(\overline{\nabla}_{V}NZ,FX).$
Again, using  (\ref{e4}) and (\ref{e6}), we obtain
$g(\nabla_{V}X,Z)\!=\!-g(\overline{\nabla}_{V}FTZ,\!X)+g(A_{NZ}V,FX).$
Here, using (\ref{e9}) and (\ref{e13})-(a), we get\\
$g(\nabla_{V}X,Z)=-\cos^{2}\!\theta
g(\nabla_{V}Z,X)-g(\overline{\nabla}_{V}NTZ,X)+g(A_{NZ}V,FX).$
According to direct calculation, we arrive at\\
$g(\nabla_{V}X,Z)=\cos^{2}\!\theta
g(\nabla_{V}X,Z)+g(A_{NTZ}V,X)+g(A_{NZ}V,FX)$ which gives
(\ref{e20}). On the other hand, for any $U,Z\in\mathcal{D}^\theta$
and $V\in\mathcal{D}^\bot$, using (\ref{e5}), (\ref{e4}) and
(\ref{e9}), we have
$g(\nabla_{U}Z,V)=g(\overline{\nabla}_{U}TZ,FV)+g(\overline{\nabla}_{U}NZ,FV).$
Hence, using (\ref{e5}) and (\ref{e4}), we obtain
$g(\nabla_{U}Z,V)=g(h(U,TZ),FV)+g(\overline{\nabla}_{U}FNZ,V).$
Here, if we use (\ref{e10}) and (\ref{e7}), we get
$g(\nabla_{U}Z,V)=g(A_{FV}U,TZ)+g(\overline{\nabla}_{U}tNZ,V)+g(\overline{\nabla}_{U}\omega
NZ,V).$ With the help of (\ref{e13})-(a), (\ref{e13})-(c),
(\ref{e12}) and
(\ref{e6}), we arrive at\\
$g(\nabla_{U}Z,V)=g(A_{FV}U,TZ)+g(\overline{\nabla}_{U}(1-\cos^{2}\!\theta)Z,V)+g(A_{NTZ}U,V).$
Upon direct calculation, we find (\ref{e21}). Similarly, we can
obtain (\ref{e22}).
\end{proof}

\section{warped product skew semi-invariant submanifolds of order $1$\\
of a locally product Riemannian manifold} In this section, we
consider a warped product submanifold of type $M\!\!=\!\!M_{1}\!\times\!_{f}\!M_{T}$ in a l.p.R. manifold $\bar{M},$ where
$M_{1}$ is a hemi-slant submanifold and $M_{T}$ is an invariant
submanifold. Then, it is clear that $M$ is a proper skew
semi-invariant submanifold of order $1$ of $\bar{M}.$ Thus, from
definition of hemi-slant submanifold and skew semi-invariant
submanifold of order $1$, we have
\begin{equation}\label{e23}
\begin{array}{c}
TM=\mathcal{D}^\theta\oplus\mathcal{D}^\bot\oplus\mathcal{D}^T
\end{array}.
\end{equation}
In particular, if $\mathcal{D}^\theta=\{0\}$, then $M$ is a warped
product semi-invariant submanifold \cite{S2}. If
$\mathcal{D}^\bot=\{0\}$, then $M$ is a warped product semi-slant
submanifold \cite{S3}.

\begin{remark} From Theorem 3.1 of \cite{S2}, we know that there are no proper
warped product semi-invariant submanifolds of type
$M_{T}\times_{f}M_{\bot}$ of a l.p.R. manifold $\bar{M}$ such that
$M_{T}$ is invariant submanifold and $M_{\bot}$ is anti-invariant
submanifold of $\bar{M}$. On the other hand, from Theorem 3.1 of
\cite{S3}, we know that there is no proper warped product
submanifold in the form $M_{T}\times_{f}M_{\theta}$ of a l.p.R.
manifold $\bar{M}$ such that $M_{\theta}$ is a proper slant
submanifold and $M_{T}$ is an invariant submanifold of $\bar{M}$.
Thus, we conclude that there is no warped product skew
semi-invariant submanifold of order $1$ of type
$M_{T}\times_{f}M_{1}$ of a l.p.R. manifold $\bar{M}$ such that
$M_{1}$ is a hemi-slant submanifold and $M_{T}$ is an invariant
submanifold of $\bar{M}$.
\end{remark}
We now present an example of warped product semi-invariant
submanifold of order $1$ of type $M_{1}\times_{f}M_{T}$ in a l.p.R.
manifold.

\begin{example} Consider the locally product Riemannian manifold $\mathbb{R}^{10}=\mathbb{R}^{5}\times\mathbb{R}^{5}$
with usual metric $g$ and almost product structure $F$ defined by
\begin{equation}\label{}\nonumber
\begin{array}{c}
F(\partial_{i})=\partial_{i},\quad F(\partial_{j})=-\partial_{j}
\end{array},
\end{equation}
where $i\in\{1,...,5\},  j\in\{6,...,10\},  \partial_{k}=\frac{\partial}{\partial x_{k}}$ and $(x_{1},...,x_{10})$
are natural coordinates of $\mathbb{R}^{10}$. Let $M$ be a submanifold of $\bar{M}=(\mathbb{R}^{10}\!,g,F)$ given by
\begin{equation}\label{}\nonumber
\begin{array}{c}
\phi(x,y,z,u,v)=(x+y,\,x-y,\,x\!\cos\!u,\,x\!\sin\!u,\,z,\,-z,\,x,\,\frac{2}{\sqrt{3}}y,\,x\!\cos\!v,\,x\!\sin\!v)
\end{array},
\end{equation}
where $x>0.$\\

Then, we easily see that the local frame of $TM$ is spanned by
\begin{equation}\label{}\nonumber
\begin{array}{c}
\phi_{x}=\partial_{1}+\partial_{2}+\cos\!u\partial_{3}+\sin\!u\partial_{4}+\partial_{7}+\cos\!v\partial_{9}+\sin\!v\partial_{10}
\end{array},
\end{equation}
\begin{equation}\label{}\nonumber
\begin{array}{c}
\phi_{y}=\partial_{1}-\partial_{2}+\frac{2}{\sqrt{3}}\partial_{8},\,\quad \phi_{z}=\partial_{5}-\partial_{6}
\end{array},
\end{equation}
\begin{equation}\label{}\nonumber
\begin{array}{c}
\phi_{u}=-x\!\sin\!u\partial_{3}+x\!\cos\!u\partial_{4},\,\quad \phi_{v}=-x\!\sin\!v\partial_{9}+x\!\cos\!v\partial_{10}
\end{array}.
\end{equation}
 Then by direct calculation, we see that $\mathcal{D}^\theta=span\{\phi_{x},\phi_{y}\}$ is a slant distribution with
 slant angle $\theta=\arccos\frac{1}{5}$ and $\mathcal{D}^\bot=span\{\phi_{z}\}$ is an anti-invariant distribution
 since $F(\phi_{z})$ is orthogonal to $TM.$ Moreover, $\mathcal{D}^T=span\{\phi_{u},\phi_{v}\}$ is an invariant distribution.
 Thus, we conclude that $M$ is a proper skew semi-invariant submanifold of order 1 of $\bar{M}$. Furthermore,
 one can easily see that $\mathcal{D}^\theta\oplus\mathcal{D}^\bot$ and $\mathcal{D}^T$ are integrable.
 If we denote the integral submanifolds $\mathcal{D}^\theta, \mathcal{D}^\bot$ and $\mathcal{D}^T$ by
 $M_{\theta}, M_{\bot}$ and $M_{T}$, respectively, then the induced metric tensor of $M$ is
\begin{equation}\label{}\nonumber
\begin{array}{c}
ds^{2}=5dx^{2}+\frac{10}{3}dy^{2}+2dz^{2}+x^{2}(du^{2}+dv^{2})\\
\,\,=g_{M_{\theta}}+g_{M_{\bot}}+x^{2}g_{M_{T}}.\qquad\qquad\qquad
\end{array}
\end{equation}
Thus, $M=(M_{\theta}\times M_{\bot})\times_{x^{2}}M_{T}$ is a warped product skew semi-invariant
submanifold of order 1 of $\bar{M}$ with warping function $f=x.$
\end{example}

Let $\mathcal{D}^\theta$ and $\mathcal{D}^T$  be slant and invariant distributions on $M$,
respectively. Then we say that $M$ is $(\mathcal{D}^\theta,\mathcal{D}^T)$ \emph{mixed totally geodesic} if $h(Z,X)\!=\!0,$
where $Z\!\in\!\mathcal{D}^\theta$ and $X\!\in\!\mathcal{D}^T$ \cite{Ro}.\\

Before giving a necessary and sufficient condition for skew
semi-invariant submanifold of order 1 to be a locally warped
product, we recall that the S. Hiepko's result \cite{Hi}, (cf.
\cite{Di}, Remark 2.1): Let $\mathcal{D}_{1}$ be a vector subbundle
in the tangent bundle of a Riemannian manifold $M$ and let
$\mathcal{D}_{2}$ be its normal bundle. Suppose that the two
distributions are involutive. If we denote by $M_{1}$ and $M_{2}$
the integral manifolds of  $\mathcal{D}_{1}$ and  $\mathcal{D}_{2}$,
respectively, then $M$ is locally isometric to warped product
$M_{1}\times_{f}M_{2}$ if the integral manifold $M_{1}$ totally
geodesic and the integral manifold $M_{2}$ is  an extrinsic sphere,
in other word, $M_{2}$ is a totally umbilical submanifold with a
parallel mean  curvature vector.
\begin{theorem}
Let $M$ be a $(\mathcal{D}^\theta,\mathcal{D}^T)$ mixed totally geodesic proper
skew semi-invariant submanifold of order 1 with integrable distribution $\mathcal{D}^T$
of a l.p.R. manifold $\bar{M}$. Then $M$ is a locally warped product submanifold if and only if
\begin{equation}\label{e24}
\begin{array}{c}
A_{FV}FX=-V[\sigma]X
\end{array},
\end{equation}
and
\begin{equation}\label{e25}
\begin{array}{c}
A_{NZ}FX+A_{NTZ}X=-Z[\sigma]\sin^{2}\!\theta X
\end{array}
\end{equation}
for $X\in\mathcal{D}^T$, $Z\in\mathcal{D}^\theta$, $V\in\mathcal{D}^\bot$
and a function $\sigma$ defined on $M$ such that $Y[\sigma]=0$ for  $Y\in\mathcal{D}^T$.
\end{theorem}

\begin{proof} Let $M=\!\!M_{1}\!\times\!_{f}\!M_{T}$ be a $(\mathcal{D}^\theta,\mathcal{D}^T)$ mixed totally geodesic warped product
proper skew semi-invariant submanifold of order 1 with integrable
distribution $\mathcal{D}^T$ of a l.p.R. manifold $\bar{M}$. Then
using (\ref{e14}) and (\ref{e16}), we have $g(A_{FV}W,FX)=0$ and
$g(A_{FV}Z,FX)=0$ for any $V,W\in\mathcal{D}^\bot$,
$Z\in\mathcal{D}^\theta$ and $X\in\mathcal{D}^T$. Since $A$ is self
adjoint, we deduce that $A_{FV}FX$ has no components in $TM_{1}$. So
$A_{FV}FX\in \mathcal{D}^T.$ Thus, using (\ref{e5}), (\ref{e4}) and
(\ref{e3}), for any $Y\in\mathcal{D}^T$, we obtain
$g(A_{FV}FX,Y)\!=\!-g(\overline{\nabla}_{Y}FV,FX)\!=\!-g(\overline{\nabla}_{Y}V,X)\!=\!-g(\nabla_{Y}V,X)=-V(\ln\!f)g(X,Y).$
Which proves (\ref{e24}). Since $M$ is
$(\mathcal{D}^\theta,\mathcal{D}^T)$ mixed totally geodesic for any
$Z\in\mathcal{D}^\theta$ and $X\!\!\in\!\mathcal{D}^T\!$, we have
$g(A_{NT\!Z}X,Z)\!=0.$ It means that $A_{NT\!Z}X\!$ has no
components in $\mathcal{D}^\theta$. On the other hand, from Lemma
3.3 of \cite{Ta}, we know that $TZ\!\in\!\mathcal{D}^\theta$ for any
$Z\!\in\!\mathcal{D}^\theta$. Thus, using this fact and (\ref{e3}),
from (\ref{e22}), we get $g(A_{NTZ}X,V)\!=0)$, that is, $A_{NTZ}X$
has no components in $\mathcal{D}^\bot.$ Thus, from (\ref{e23}), we
conclude that $A_{NTZ}X\in\mathcal{D}^T$. Also, we have
$A_{NZ}X\in\mathcal{D}^T$. Then, for $X,Y\in\mathcal{D}^T$ and
$Z\in\mathcal{D}^\theta$, with the help of (\ref{e3}), from
(\ref{e18}), we have ${g(A_{NTZ}Y,X)+g(A_{NZ}FY,X)=-\sin^{2}\!\theta
g(\nabla_{X}Z,Y)}=-\sin^{2}\!\theta Z(\ln\!f)g(Y,X)$. This proves
(\ref{e25}). Moreover, $Y(\ln\!f)=0$ for a warped product
proper skew semi-invariant submanifold of order 1, we obtain $\sigma=\ln\!f.$\\

Conversely, suppose that  $M$ is
$(\mathcal{D}^\theta,\mathcal{D}^T)$ mixed totally geodesic proper
skew semi-invariant submanifold of order 1 with integrable
distribution $\mathcal{D}^T$ of a l.p.R. manifold $\bar{M}$ such
that (\ref{e24}) and (\ref{e25}) hold. We know from Theorem 4.6 of
\cite{Ta}, $\mathcal{D}^\bot$ is always integrable. So, we have
$g(\nabla_{V}W,X)=0$ for $V,W\in\mathcal{D}^\bot$ and
$X\in\mathcal{D}^T$. Using this fact, (\ref{e24}), (\ref{e25}) and
(\ref{e15})-(\ref{e17}), it is not difficult to see that $M_{1}$ is
totally geodesic in $M$. Let $M_{T}$ be the integral manifold of
$\mathcal{D}^T$ and $h_{T}$ be the second fundamental form of
$M_{T}$ in $M$. From (\ref{e5}), we have
$g(h_{T}(X,Y),V)=g(\nabla_{X}Y,V)$ for $X,Y\in\mathcal{D}^T$ and
$V\in\mathcal{D}^\bot.$ Then, (\ref{e19}) imply that
$g(h_{T}(X,Y),V)=g(A_{FV}FY,X)$. Thus, using (\ref{e24}), we obtain
\begin{equation}\label{e26}
\begin{array}{c}
g(h_{T}(X,Y),V)=-V[\sigma]g(Y,X)
\end{array}.
\end{equation}
Similarly, from (\ref{e5}), we have
$g(h_{T}(X,Y),Z)=g(\nabla_{X}Y,Z)$ for $X,Y\in\mathcal{D}^T$ and
$Z\in\mathcal{D}^\theta.$ Using (\ref{e18}), we obtain
$g(h_{T}(X,Y),Z)=\csc^{2}\!\theta \{g(A_{NTZ}Y,X)+g(A_{NZ}FY,X)\}.$
Thus, from (\ref{e25}), we get
\begin{equation}\label{e27}
\begin{array}{c}
g(h_{T}(X,Y),Z)=-Z[\sigma]g(X,Y)
\end{array}.
\end{equation}
Thus, for any $E=V+Z\in TM_{1}$, from (\ref{e26}) and (\ref{e27}), we arrive at
\begin{equation}\label{e28}
\begin{array}{c}
g(h_{T}(X,Y),E)=g(h_{T}(X,Y),V)+g(h_{T}(X,Y),Z)\\
\qquad\quad=-\{V[\sigma]+Z[\sigma]\}g(X,Y).
\end{array}
\end{equation}
Last equation (\ref{e28}) says that $M_{T}$ is totally umbilical in
$M.$ Let denote by $grad^{\bot}\sigma$ and $grad^{\theta}\sigma$ the
gradient of $\sigma$ on $\mathcal{D}^\bot$ and $\mathcal{D}^\theta$,
respectively. From (\ref{e28}), we write
\begin{equation}\label{e29}
\begin{array}{c}
h_{T}(X,Y)=-\{grad^{\bot}\sigma+grad^{\theta}\sigma\}g(X,Y).
\end{array}
\end{equation}
Thus, for any $E=V+Z\in TM_{1}$, we have\\

$\!\!\!\!\!g(\nabla_{X}(grad^{\bot}\sigma+grad^{\theta}\sigma),E)=g(\nabla_{X}grad^{\bot}\sigma,E)+g(\nabla_{X}grad^{\theta}\sigma,E)$\\

$\!\!\!\!\!=\{Xg(grad^{\bot}\sigma,V)-g(grad^{\bot}\sigma,\nabla_{X}E)\}$\\

$\!\!\!\!\!+\{Xg(grad^{\theta}\sigma,Z)-g(grad^{\theta}\sigma,\nabla_{X}E)\}$\\

$\!\!\!\!\!=X[V[\sigma]]-g(grad^{\bot}\sigma,\nabla_{X}E)+X[Z[\sigma]]-g(grad^{\theta}\sigma,\nabla_{X}E)$.
On the other hand, if we use (\ref{e25}) in (\ref{e22}), then we get
$g(\nabla_{X}V,Z)=-g(\nabla_{X}Z,V)=0.$ Using this fact, we obtain\\

$\!\!\!\!\!\!g(\nabla_{X}\!(grad^{\bot}\sigma+grad^{\theta}\!\sigma),\!E)\!=\!X\![V\![\sigma]]-g(grad^{\bot}\sigma,\!\nabla_{X}\!Z)+\!X\![Z\![\sigma]]-
g(grad^{\theta}\!\sigma,\!\nabla_{X}\!V)$ Upon direct calculation,
we arrive at\\

$\!\!\!\!\!g(\nabla_{X}\!(grad^{\bot}\sigma+grad^{\theta}\!\sigma),E)=\{\!X\![Z\![\sigma]]-[X,Z][\sigma]+g(grad^{\bot}\sigma,\!\nabla_{Z}\!X)\}$\\

$\!\!\!\!\!+\{X\![V\![\sigma]-[X,V][\sigma]+g(grad^{\theta}\!\sigma,\!\nabla_{V}\!X)\}$.
After some calculation, we get\\

$\!\!\!\!\!g(\nabla_{X}\!(grad^{\bot}\sigma+grad^{\theta}\!\sigma),E)$\\

$\!\!\!\!\!=\{Z[X[\sigma]]+g(grad^{\bot}\sigma,\!\nabla_{Z}X)+V[X[\sigma]]+g(grad^{\theta}\!\sigma,\!\nabla_{V}X)\}$.\\

\!\!\!\!\!Since $X[\sigma]=0$, from  the last equation, we derive\\

$\!\!\!\!\!g(\nabla_{X}\!(grad^{\bot}\sigma+grad^{\theta}\!\sigma),E)=-g(\nabla_{Z}grad^{\bot}\sigma,X)-g(\nabla_{V}grad^{\theta}\sigma,X)$.\\

Here, we know that $\nabla_{Z}grad^{\bot}\sigma,
\nabla_{V}grad^{\theta}\sigma\in TM_{1}$, since $M_{1}$ is totally
geodesic. Hence, we obtain
$g(\nabla_{X}\!(grad^{\bot}\sigma+grad^{\theta}\!\sigma),E)=0$. It
means that $grad^{\bot}\sigma+grad^{\theta}\!\sigma$ is parallel in
$M.$ This fact and (\ref{e29}) imply that $M_{T}$ is an extrinsic
sphere. This completes the proof.
\end{proof}

\section{a chen-type inequality for warped product\\ skew semi-invariant submanifolds of order $1$}

In this section, we prove that the invariant distribution which is
involved in the definition of the warped product proper skew
semi-invariant submanifolds of order $1$ of a l.p.R. manifold is
integrable under some restrictions. We also give an inequality
similar to Chen's inequality \cite{Chenn} for the squared norm of
the second fundamental form in terms of the warping function for
such submanifolds. We first give the following two lemmas for later
use.

\begin{lemma} Let $M=M_{1}\times_{f}M_{T}$ be a warped product proper skew semi-invariant
submanifold of order $1$ of a l.p.R. manifold $\bar{M}.$ Then we
have,
\begin{equation}\label{e30}
\begin{array}{c}
g(h(X,V),FW)=0
\end{array}
\end{equation}
and
\begin{equation}\label{e31}
\begin{array}{c}
g(h(X,V),NZ)=0
\end{array},
\end{equation}

for $X\in\mathcal{D}^T$, $Z\in\mathcal{D}^\theta$  and
$V,W\in\mathcal{D}^\bot$.
\end{lemma}

\begin{proof} For any $V,W\in\mathcal{D}^\bot$ and $X\in\mathcal{D}^T$,  using (\ref{e5}), (\ref{e4}) and (\ref{e3}),
we get
$g(h(X,V),FW)\!\!=\!g(\overline{\nabla}_{V}\!X,FW)\!\!=\!g(\overline{\nabla}_{V}\!F\!X,\!W)\!\!=\!g(\nabla_{V}\!F\!X,\!W)\!\!=\!\!V\!(\ln\!
f)g(F\!X,W)\!\!=\!0$,  since $g(FX,W)=0.$ Hence (\ref{e30}) follows.
In a similar way, using (\ref{e5}), (\ref{e4}), (\ref{e9})  and
(\ref{e3}), we have
\begin{equation}\label{}\nonumber
\begin{array}{c}
g(h(X,V),NZ)=g(\overline{\nabla}_{V}X,NZ)=g(\overline{\nabla}_{V}X,FZ)-g(\overline{\nabla}_{V}X,TZ)\\
=g(\overline{\nabla}_{V}FX,Z)-g(\overline{\nabla}_{V}X,TZ)\\
=g(\nabla_{V}FX,Z)-g(\nabla_{V}X,TZ)\\
\qquad\qquad\quad= V(\ln\! f)g(FX,Z)-V(\ln\! f)g(X,TZ)=0,
\end{array}
\end{equation}
since $g(FX,Z)=0$ and $g(X,TZ)=0$.
\end{proof}

\begin{lemma} Let $M=M_{1}\times_{f}M_{T}$ be a warped product proper skew semi-invariant
submanifold of order $1$ of a l.p.R. manifold $\bar{M}.$ Then we
have,
\begin{equation}\label{e32}
\begin{array}{c}
g(h(X,FY),FV)=-V(\ln\! f)g(X,Y)
\end{array}
\end{equation}
and
\begin{equation}\label{e33}
\begin{array}{c}
g(h(X,Y),NZ)=TZ(\ln\! f)g(X,Y)
\end{array}
\end{equation}
for $X,Y\in\mathcal{D}^T$, $Z\in\mathcal{D}^\theta$  and
$V\in\mathcal{D}^\bot$.
\end{lemma}

\begin{proof}
Using (\ref{e5}) and (\ref{e4}), we have\\
$g(h(X,FY),FV)=g(\overline{\nabla}_{X}FY,FV)=g(\overline{\nabla}_{X}Y,V)=g(\nabla_{X}Y,V)=-g(\nabla_{X}V,Y)$
for any  $X,Y\!\!\in\!\!\mathcal{D}^T$ and
$V\!\!\in\!\!\mathcal{D}^\bot$. Hence, using (\ref{e3}), we get
easily (\ref{e32}). Last assertion (\ref{e33}) follows from Lemma
3.1-(ii) of \cite{Al} by using linearity.
\end{proof}
\begin{theorem}
Let $M=M_{1}\times_{f}M_{T}$ be an $(q+m)$-dimensional warped
product proper skew semi-invariant submanifold of order $1$ of a
l.p.R. manifold $\bar{M}$ of dimension $2q+m,$  where $dim(M_{1})=q$
and $dim(M_{T})=m.$ Then the invariant distribution $\mathcal{D}^T$
of $M_{T}$ is integrable.
\end{theorem}

\begin{proof}
For any $X,Y\in\mathcal{D}^T$, $Z\in\mathcal{D}^\theta$  and
$V\in\mathcal{D}^\bot$, using  (\ref{e32}) and (\ref{e33}), we get
$g(h(X,FY),FV)=g(h(FX,Y),FV)$ and $g(h(X,FY),NZ)=g(h(FX,Y),NZ)$,
since $g(X,FY)=g(FX,Y)$. Hence, we conclude that $h(X,FY)=h(FX,Y)$,
since $T^{\bot}M=F\mathcal{D}^\bot\oplus N\mathcal{D}^\theta,$ where
$T^{\bot}M$ is the normal bundle of $M$ in  $\bar{M}$. Thus, our
assertion immediately comes from Theorem 1 of \cite{Bej}.
\end{proof}

Let $M$ be a $(k+n+m)$-dimensional warped product proper skew
semi-invariant submanifold of order $1$ of a $(2k+2n+m)$-dimensional
l.p.R. manifold $\bar{M}.$ We choose a canonical orthonormal basis
$\{e_{1},...,e_{m},\bar{e}_{1},...,\bar{e}_{k},\tilde{e}_{1},...,\tilde{e}_{n},e_{1}^{*},...,e_{k}^{*},F\tilde{e}_{1},...,F\tilde{e}_{n}\}$
such that $\{e_{1},...,e_{m}\}$ is an orthonormal basis of
$\mathcal{D}^T$, $\{\bar{e}_{1},...,\bar{e}_{k}\}$ is an orthonormal
basis of  $\mathcal{D}^\theta$,
$\{\tilde{e}_{1},...,\tilde{e}_{n}\}$ is an orthonormal basis of
$\mathcal{D}^\bot$, $\{e_{1}^{*},...,e_{k}^{*}\}$ is an orthonormal
basis of $N\mathcal{D}^\theta$ and
$\{F\tilde{e}_{1},...,F\tilde{e}_{n}\}$ is an orthonormal basis of
$F\mathcal{D}^\bot$.
\begin{remark} In view of (\ref{e4}), we can observe that $\{Fe_{1},...,Fe_{m}\}$ is
also an orthonormal basis of $\mathcal{D}^T$. On the other hand,
with the help of the equations (3.5) and (3.6) of \cite{Ta}, we can
see that $\{\sec\!\theta T\bar{e}_{1},...,\sec\!\theta
T\bar{e}_{k}\}$ is also an orthonormal basis of $\mathcal{D}^\theta$
and $\{\csc\!\theta N\bar{e}_{1},...,\csc\!\theta N\bar{e}_{k}\}$ is
also an orthonormal basis of $N\mathcal{D}^\theta$.
\end{remark}

We now state the main result of this section.

\begin{theorem}
Let $M=M_{1}\times_{f}M_{T}$ be a  $(k+n+m)$-dimensional warped
product proper skew semi-invariant submanifold of order $1$ of a
$(2k+2n+m)$-dimensional l.p.R. manifold $\bar{M}.$ Then the squared
norm of the second fundamental form of $M$ satifies
\begin{equation}\label{e34}
\begin{array}{c}
\|h\|^{2}\geq
m\{\|\nabla^{\bot}(\ln\!f)\|^{2}+\cot^{2}\!\theta\|\nabla^{\theta}(\ln\!f)\|^{2}\}
\end{array},
\end{equation}
where $m=dim(M_{T})$, $\nabla^{\bot}(\ln\!f)$ and
$\nabla^{\theta}(\ln\!f)$ are gradients of $\ln\!f$ on
$\mathcal{D}^\bot$
 and $\mathcal{D}^\theta,$ respectively. If the equality case of (\ref{e34}) holds, then $M_{1}$ is a
 totally geodesic submanifold of $\bar{M}$ and $M$ is mixed totally geodesic.
 Moreover, $M_{T}$ can not be minimal.
\end{theorem}

\begin{proof} In view of decomposition (\ref{e23}), the squared norm of the
second fundamental form $h$ can be decomposed as
\begin{equation}\label{}\nonumber
\begin{array}{c} \|h\|^{2}=\|h(\mathcal{D}^T,\mathcal{D}^T)\|^{2}+\|h(\mathcal{D}^\theta,\mathcal{D}^{\theta})\|^{2}
+\|h(\mathcal{D}^\bot,\mathcal{D}^\bot)\|^{2}\\
\qquad\qquad+2\|h(\mathcal{D}^T,\mathcal{D}^\bot)\|^{2}+2\|h(\mathcal{D}^T,\mathcal{D}^{\theta})\|^{2}+2\|h(\mathcal{D}^\bot,\mathcal{D}^{\theta})\|^{2}.
\end{array}
\end{equation}
Which can be written as follows:
\begin{equation}\label{e35}
\begin{array}{c}
\|h\|^{2}=\displaystyle\sum^{m}_{i,j=1}\displaystyle\sum^{n}_{a=1}g(h(e_{i},e_{j}),F\tilde{e}_{a})^{2}+
\displaystyle\sum^{m}_{i,j=1}\displaystyle\sum^{k}_{r=1}g(h(e_{i},e_{j}),e_{r}^{*})^{2}\\
\quad+\displaystyle\sum^{n}_{a,b,c=1}\!\!\!g(h(\tilde{e}_{a},\tilde{e}_{b}),F\tilde{e}_{c})^{2}+
\displaystyle\sum^{n}_{a,b=1}\displaystyle\sum^{k}_{r=1}g(h(\tilde{e}_{a},\tilde{e}_{b}),e_{r}^{*})^{2}\\
\qquad+\displaystyle\sum^{k}_{r,s=1}\displaystyle\sum^{n}_{a=1}g(h(\bar{e}_{r},\bar{e}_{s}),F\tilde{e}_{a})^{2}+
\displaystyle\sum^{k}_{r,s,q=1}g(h(\bar{e}_{r},\bar{e}_{s}),e_{q}^{*})^{2}\\
\qquad\qquad\quad+2\displaystyle\sum^{m}_{i=1}\displaystyle\sum^{n}_{a,b=1}g(h(e_{i},\tilde{e}_{a}),F\tilde{e}_{b})^{2}+
2\displaystyle\sum^{m}_{i=1}\displaystyle\sum^{n}_{a=1}\displaystyle\sum^{k}_{r=1}g(h(e_{i},\tilde{e}_{a}),e_{r}^{*})^{2}\\
\qquad\qquad\quad+2\displaystyle\sum^{m}_{i=1}\displaystyle\sum^{k}_{r=1}\displaystyle\sum^{n}_{a=1}g(h(e_{i},\bar{e}_{r}),F\tilde{e}_{a})^{2}+
2\displaystyle\sum^{m}_{i=1}\displaystyle\sum^{k}_{r,s=1}g(h(e_{i},\bar{e}_{r}),e_{s}^{*})^{2}\\
\qquad\qquad+2\displaystyle\sum^{k}_{r=1}\displaystyle\sum^{n}_{a,b=1}g(h(\bar{e}_{r},\tilde{e}_{a}),F\tilde{e}_{b})^{2}+
2\displaystyle\sum^{k}_{r,s=1}\displaystyle\sum^{n}_{a=1}g(h(\bar{e}_{r},\tilde{e}_{a}),e_{s}^{*})^{2}.
\end{array}
\end{equation}
Here, using (\ref{e30})-(\ref{e32}) and Remark 5.4, we have
\begin{equation}\label{e36}
\begin{array}{c}\displaystyle\sum^{m}_{i,j=1}\displaystyle\sum^{n}_{a=1}g(h(e_{i},e_{j}),F\tilde{e}_{a})^{2}
=\displaystyle\sum^{m}_{i,j=1}\displaystyle\sum^{n}_{a=1}(-\tilde{e}_{a}(\ln
\!f)g(e_{i},e_{j}))^{2}
\end{array}
\end{equation}
and
\begin{equation}\label{e37}
\begin{array}{c}\displaystyle\sum^{m}_{i,j=1}\displaystyle\sum^{k}_{r=1}g(h(e_{i},e_{j}),e_{r}^{*})^{2}
=\displaystyle\sum^{m}_{i,j=1}\displaystyle\sum^{k}_{r=1}g(h(e_{i},e_{j}),N\bar{e}_{r})^{2}\!\csc^{2}\!\theta.
\end{array}
\end{equation}
Also, using (\ref{e33}) from (\ref{e37}), we get
\begin{equation}\label{e38}
\begin{array}{c}
\displaystyle\sum^{m}_{i,j=1}\displaystyle\sum^{k}_{r=1}g(h(e_{i},e_{j}),e_{r}^{*})^{2}
=\displaystyle\sum^{m}_{i,j=1}\displaystyle\sum^{k}_{r=1}(T\bar{e}_{r}(\ln
\!f)g(e_{i},e_{j}))^{2}\!\csc^{2}\!\theta.
\end{array}
\end{equation}
Using (\ref{e36}) and (\ref{e38}) from (\ref{e35}), we get
\begin{equation}\label{e39}
\begin{array}{c}
\|h\|^{2}\geq
m\|\nabla^{\bot}(\ln\!f)\|^{2}+\displaystyle\sum^{m}_{i,j=1}\displaystyle\sum^{k}_{r=1}(T\bar{e}_{r}(\ln
\!f)g(e_{i},e_{j}))^{2}\!\csc^{2}\!\theta
\end{array}.
\end{equation}
In view of Remark 5.4, we replace $\bar{e}_{r}$ by $\sec\!\theta
T\bar{e}_{r}$ in the last term of (\ref{e39}) and using (\ref{e12}),
we have
\begin{equation}\label{e40}
\begin{array}{c}
\displaystyle\sum^{m}_{i,j=1}\displaystyle\sum^{k}_{r=1}(T\bar{e}_{r}(\ln
\!f)g(e_{i},e_{j}))^{2}\!\csc^{2}\!\theta\qquad\qquad\qquad\qquad\qquad\\
\qquad=\!\!\displaystyle\sum^{m}_{i,j=1}\displaystyle\sum^{k}_{r=1}\cos^{4}\!\theta(\bar{e}_{r}(\ln
\!f)g(e_{i},e_{j}))^{2}\!\csc^{2}\!\theta)
=m\cot^{2}\!\theta\|\nabla^{\theta}(\ln\!f)\|^{2}.
\end{array}
\end{equation}
Thus, using (\ref{e40}) in (\ref{e39}), we find (\ref{e34}).\\

Next, if the equality case of (\ref{e34}) holds, then from
(\ref{e35}), we have
\begin{equation}\label{e41}
\begin{array}{c}
h(\mathcal{D}^\bot,\mathcal{D}^\bot)=0, \quad
h(\mathcal{D}^\theta,\mathcal{D}^\theta)=0, \quad
h(\mathcal{D}^\bot,\mathcal{D}^\theta)=0
\end{array}
\end{equation}
and
\begin{equation}\label{e42}
\begin{array}{c}
h(\mathcal{D}^T,\mathcal{D}^\bot)=0,\quad
h(\mathcal{D}^T,\mathcal{D}^\theta)=0.
\end{array}
\end{equation}
Since $M_{1}$ is totally geodesic in $M$, from (\ref{e41}) it
follows that $M_{1}$ is also totally geodesic in $\bar{M}.$ On the
other hand (\ref{e42}) imply that $M$ is mixed totally geodesic.
Finally, if we suppose that  $M$ is minimal, then from (\ref{e32})
and (\ref{e33}), we conclude that $\|\nabla(\ln\!f)\|=0$, which is a
contradiction.
\end{proof}
\begin{remark} Theorem 5.5 coincides with Theorem 4.2 of \cite{S2} in case
$\mathcal{D}^\theta=\{0\}$. In other word, Theorem 5.5 is a
generalization of Theorem 4.2 of \cite{S2}.
\end{remark}

\bibliographystyle{amsplain}

\end{document}